\documentclass[preprint]{imsart}

\pdfoutput=1
\RequirePackage[OT1]{fontenc}
\usepackage{amsthm, amsmath, natbib, amssymb, bm, enumitem}
\RequirePackage[colorlinks,citecolor=blue,urlcolor=blue]{hyperref}

\usepackage{mathtools} 
\mathtoolsset{showonlyrefs}

\usepackage{tikz}
\usepackage{pgfplotstable}
\usetikzlibrary{arrows,positioning,plotmarks,external,patterns,angles,
	decorations.pathmorphing,backgrounds,fit,shapes,graphs,calc,spy}
\pgfplotsset{compat=1.14}

\startlocaldefs
\numberwithin{equation}{section}
\theoremstyle{plain}
\newtheorem{thm}{Theorem}[section]
\newtheorem{lem}{Lemma}[section]
\newtheorem{corollary}{Corollary}[section]
\newcommand{\rd}{\mathrm{d}}
\def\citeapos#1{\citeauthor{#1}'s (\citeyear{#1})}
\endlocaldefs

\begin{document}
\begin{frontmatter}
\title{Minimaxity under half-Cauchy type priors}
\runtitle{Minimaxity under half-Cauchy type priors}

\begin{aug}
\author{\fnms{Yuzo} \snm{Maruyama}\thanksref{addr1,t1}\ead[label=e1]{maruyama@port.kobe-u.ac.jp}}
\and
\author{\fnms{Takeru} \snm{Matsuda}\thanksref{addr2,t2}\ead[label=e2]{matsuda@mist.i.u-tokyo.ac.jp}}

\runauthor{Y.~Maruyama and T.~Matsuda}

\address[addr1]{Kobe University \printead{e1} 
}

\address[addr2]{The University of Tokyo \& RIKEN Center for Brain Science \printead{e2}
}

%Department of Mathematical Informatics, Graduate School of Information Science and Technology, 
%The University of Tokyo & Statistical Mathematics Unit, RIKEN Center for Brain Science

\thankstext{t1}{supported by JSPS KAKENHI Grant Numbers 19K11852 and 22K11933}
\thankstext{t2}{supported by JSPS KAKENHI Grant Numbers 21H05205, 22K17865 and JST Moonshot Grant Number JPMJMS2024}
\end{aug}

\begin{abstract}
This is a follow-up paper of Polson and Scott (2012, Bayesian Analysis),
%which was mainly interested in Bayesian inference with the half-Cauchy prior.
which claimed that the half-Cauchy prior is a sensible default prior for a scale parameter in hierarchical models. 
%The half-Cauchy prior has a hierarchical structure and
%As Polson and Scott pointed out, 
%the second stage prior on a scale parameter has the U-shape which 
For estimation of a normal mean vector under the quadratic loss, they showed that the Bayes estimator with respect to the half-Cauchy prior seems to be minimax through numerical experiments.
In terms of the shrinkage coefficient, the half-Cauchy prior has a U-shape and can be interpreted as a continuous spike and slab prior.
In this paper, we consider a general class of priors with U-shapes and theoretically establish sufficient conditions for the minimaxity of the corresponding (generalized) Bayes estimators.
We also develop an algorithm for posterior sampling and present numerical results.
\end{abstract}

\begin{keyword}[class=MSC]
\kwd[Primary ]{62C20}
%\kwd{}
%\kwd[; secondary ]{}
\end{keyword}

\begin{keyword}
\kwd{minimaxity}
\kwd{shrinkage}
\kwd{spike and slab prior}
\end{keyword}

\end{frontmatter}

\section{Introduction}
\label{sec:intro}
Consider a normal hierarchical model 
\begin{align}
y\mid\beta&\sim N_p(\beta,I_p), \label{hierarchical.1} \\
\beta\mid\kappa &\sim N_p(0,\kappa^{-1}(1-\kappa) I_p), \label{hierarchical.2} \\
\kappa &\sim \pi (\kappa), \label{hierarchical.3}
\end{align}
where the hyperparameter $\kappa \in (0,1)$ specifies the shrinkage coefficient of the Bayes estimator (posterior mean) of $\beta$:
\begin{align}
	\hat{\beta}(y) = {\rm E} [\beta \mid y] = (1-{\rm E}[\kappa \mid y]) y.
\end{align}
\cite{Polson-Scott-2012} claimed that the hyperprior
\begin{equation}\label{prior.half}
 \pi(\kappa) \propto \kappa^{-1/2}(1-\kappa)^{-1/2},
\end{equation}
is a sensible default choice.
Since it has a U-shape with
\begin{equation}\label{cont.spike.slab}
 \lim_{\kappa\to 0}\pi(\kappa)=\lim_{\kappa\to 1}\pi(\kappa)=\infty,
\end{equation}
it may be regarded as a continuous spike and slab prior.
See \cite{carvalho2010horseshoe} for a related discussion in the context of horseshoe priors.
For the parameterization
\begin{equation}
 \lambda=\sqrt{\frac{1-\kappa}{\kappa}} \in (0,\infty),
\end{equation}
the prior \eqref{prior.half} is expressed as
\begin{equation}
 \pi(\lambda) \propto \frac{1}{1+\lambda^2}I_{(0,\infty)}(\lambda),
\end{equation}
which is the reason why the prior \eqref{prior.half} is called the half-Cauchy prior.

%This is a follow-up paper of \cite{Polson-Scott-2012},
%which was mainly interested in Bayesian inference with the half-Cauchy prior.
%They discussed several advantages of the half-Cauchy prior \eqref{prior.half}. 
\cite{Polson-Scott-2012} introduced a class of ``hypergeometric inverted-beta priors'' with the density
\begin{equation}\label{prior.hypergeometric}
 \pi(\kappa) \propto \kappa^{a-1}(1-\kappa)^{b-1}(1+c\kappa)^{-1}\exp(d\kappa),
\end{equation}
which is a generalization of the half-Cauchy prior \eqref{prior.half}.
For estimation of $p$-variate normal mean $\beta$ under the quadratic loss $\|\hat{\beta}-\beta\|^2$, they derived expressions for the risk of Bayes estimators with respect to the prior \eqref{prior.hypergeometric}. 
Recall that the usual estimator $\hat{\beta}=y$ is inadmissible for $p\geq 3$ although it is minimax for any $p$ \citep{Stein-1974, DSW-2018}.
Through numerical experiments, \cite{Polson-Scott-2012} discussed the minimaxity of the Bayes estimators under \eqref{prior.hypergeometric} for $p\geq 3$ and compared them with the James--Stein estimator.

In this paper, we consider more general (possibly improper) prior 
\begin{equation}\label{our.prior}
 \pi(\kappa)= \kappa^{a-1}(1-\kappa)^{b-1}h(\kappa),
\end{equation}
where we assume
\begin{enumerate}[label= \textbf{A.\arabic*}, leftmargin=*]
\item\label{A.1} [around $\kappa=0$] \ $a<1$ and $h(\kappa)$ is slowly varying at $\kappa=0$ with
\begin{equation}\label{hh}
 \lim_{\kappa\to 0}\kappa\frac{h'(\kappa)}{h(\kappa)}=0,
\end{equation}
\item\label{A.2} [around $\kappa=1$] \ $0<b<1$ and $h(1)<\infty$.
\end{enumerate}
We derive sufficient conditions for the minimaxity of the (generalized) Bayes estimator with respect to the prior \eqref{our.prior} and give several examples.
By \ref{A.1} and \ref{A.2}, the prior \eqref{our.prior} has a U-shape with \eqref{cont.spike.slab} and 
is hence regarded as a continuous spike and slab prior.
More properties of the prior and the corresponding marginal density are given in Section \ref{sec:prior}.
%The main contribution of this paper is to propose minimax Bayes estimators
%under the priors $\pi(\kappa)$.
%With $-p/2<a<1$ and $0<b<1$, the prior density with satisfies
%\begin{align}
% \lim_{\kappa\to 0}\pi(\kappa)=\lim_{\kappa\to 1}\pi(\kappa)=\infty.
%\end{align}
%Clearly, given $b\in(0,1)$, 
%the prior with $a< 0$ and mostly with $a=0$ is improper.
%Since we are mainly interested in minimaxity of the (generalized) Bayes estimators, 
%which is well-defined for $a>-p/2$, as in \eqref{marginal.0}.
%
%we provide the theoretical properties, minimaxity of the Bayes estimators, 
%as well as some guidance for numerical implementation.
%
\cite{Fourdrinier-etal-1998}
showed the minimaxity of (generalized) Bayes estimators 
under priors similar to \eqref{our.prior}, but with $b\geq 1$ in \ref{A.2}. 
%However, our prior, \eqref{our.prior} with \ref{A.1} and \ref{A.2}, is not included 
%in the class of priors they treated.
%As shown in Section \ref{sec.minimaxity}, we need to be careful
%with $0<b<1$ when we apply an integration by parts. See \eqref{key.derivative},
%\eqref{int.part.1}, and \eqref{int.part.2} for a new type of the integration by parts.
Furthermore, the first author of this paper, \cite{Maruyama-1998}, established the minimaxity for the case $h(\kappa)\equiv 1$, 
\begin{equation}
 \pi(\kappa)= \kappa^{a-1}(1-\kappa)^{b-1}
\end{equation}
with some $b\in(0,1)$.
Hence this paper can be also viewed as an extension of two papers 
\cite{Fourdrinier-etal-1998} and \cite{Maruyama-1998}.

The organization of this paper is as follows. 
In Section \ref{sec:prior}, we investigate the properties of the prior \eqref{our.prior} and the corresponding marginal density.
In Section \ref{sec.minimaxity}, we derive sufficient conditions for the minimaxity of the (generalized) Bayes estimator with respect to the prior \eqref{our.prior} with \ref{A.1} and \ref{A.2}.
In Section \ref{sec.example}, we give several examples of the prior \eqref{our.prior} for which the (generalized) Bayes estimator is minimax.
While the half-Cauchy prior itself is not included in these examples, we provide a variant of the half-Cauchy prior that gives a minimax Bayes estimator.
%In this paper, the minimaxity of the Bayes estimators under some priors with U-shape 
%\eqref{cont.spike.slab} is theoretically established.
In Section \ref{sec:simulation}, we develop an algorithm for sampling from the posterior distribution under the prior \eqref{our.prior} and present some numerical results.

%Among them, we suggest a proper prior on the
%boundary between propriety and impropriety
%\begin{equation}\label{proper.prior.boundary}
% \pi(\kappa)\propto \kappa^{-1}(1-\kappa)^{b-1}\{1+c_1\log(1/\kappa)\}^{-2},
%\end{equation}
%which was not investigated in \cite{Polson-Scott-2012}.
%\begin{equation}
% \int_0^1\kappa^{-1}(1-\kappa)^{b-1}\frac{\rd \kappa}{\kappa \{1+c_1\log(1/\kappa)\}^2}=
%\frac{1}{c_1}\left[\frac{1}{1+c_1\log(1/\kappa)}\right]_0^1=\frac{1}{c_1}.
%\end{equation}

\section{Properties of the prior and the marginal density}
\label{sec:prior}
%As in \eqref{our.prior}, the prior on $\kappa$ is
%given by
%\begin{align}
%  \pi(\kappa)\propto\kappa^{a-1}(1-\kappa)^{b-1}h(\kappa).
%\end{align}
%Then %Under the prior \eqref{our.prior} on $\kappa$, 
For the hierarchical model \eqref{hierarchical.1}, \eqref{hierarchical.2} and \eqref{hierarchical.3}, the prior on $\beta$ is given by
\begin{equation}\label{our.prior.beta}
 \pi(\beta)=\int_0^1 \frac{1}{(2\pi)^{p/2}}\left(\frac{\kappa}{1-\kappa}\right)^{p/2}\exp\left(-\frac{\kappa}{1-\kappa}\frac{\|\beta\|^2}{2}\right)\pi(\kappa)\rd \kappa,
\end{equation}
and the marginal density of $y$ is given by
\begin{equation}
  m(y)=\int \frac{1}{(2\pi)^{p/2}}\exp\left(-\frac{\|y-\beta\|^2}{2}\right)\pi(\beta)\rd \beta. 
\end{equation}
The following result summarizes the properties of $\pi(\kappa)$ in \eqref{our.prior}
satisfying \ref{A.1} and \ref{A.2} and $m(y)$.
\begin{lem}\label{lem.prop.prior}
\begin{enumerate}
% \item \label{lem.prop.prior.1}$\lim_{\kappa\to 1}\pi(\kappa)=\infty$.
% \item \label{lem.prop.prior.2}$\lim_{\kappa\to 0}\pi(\kappa)=\infty$ if either $a<1$ or \{$a=1$ and
%$\lim_{\kappa\to 0}h(\kappa)=\infty$\}.
\item \label{lem.prop.prior.3}The prior $\pi(\kappa)$ in \eqref{our.prior} is proper if either $a>0$ or 
\{$a=0$ and $\int_0^1 \kappa^{-1}h(\kappa)\rd \kappa <\infty$\}.
%\item \label{lem.prop.prior.4}The prior on $\beta$, $\pi(\beta)$, satisfies
%\begin{align}
% \lim_{\|\beta\|\to 0}\frac{\pi(\beta)}{\|\beta\|^{-p+2b}}&=h(1)\frac{\Gamma(p/2-b)2^{p/2-b}}{(2\pi)^{p/2}},\\
% \lim_{\|\beta\|\to \infty}\frac{\pi(\beta)}{\|\beta\|^{-p-2a}h(1/\|\beta\|)}&=\frac{\Gamma(p/2+a)2^{p/2+a}}{(2\pi)^{p/2}}.
%\end{align}
\item \label{lem.prop.prior.5}The marginal density of $y$ with respect to \eqref{our.prior} is 
\begin{equation}\label{my.1}
 m(y)=
\frac{1}{(2\pi)^{p/2}}
\int_0^1\exp\left(-\kappa\frac{\|y\|^2}{2}\right)\kappa^{p/2+a-1}(1-\kappa)^{b-1}h(\kappa)\rd\kappa,
\end{equation}
which is finite for every $y\in\mathbb{R}^p$ if either $a>-p/2 $ or 
\{$a=-p/2$ and $\int_0^1 \kappa^{-1}h(\kappa)\rd \kappa <\infty$\}.
\end{enumerate} 
\end{lem}
\begin{proof}
{} %[Parts \ref{lem.prop.prior.1} \& \ref{lem.prop.prior.2}] 
%They are straightforward from the form of $\pi(\kappa)$ with \ref{A.1} and \ref{A.2}.
[Part \ref{lem.prop.prior.3}]
By \ref{A.2}, $\pi(\kappa)$ is integrable around $\kappa=1$, since 
\begin{equation}
 \lim_{\kappa\to 1}\frac{\pi(\kappa)}{h(1)(1-\kappa)^{b-1}}=1.
\end{equation}
By \ref{A.1},  $\pi(\kappa)$ satisfies
\begin{equation}
 \lim_{\kappa\to 0}\frac{\pi(\kappa)}{\kappa^{a-1}h(\kappa)}=1,
\end{equation}
and $\kappa^{a-1}h(\kappa)$ is integrable around $\kappa=0$ if either $a>0$ or 
\{$a=0$ and $\int_0^1 \kappa^{-1}h(\kappa)\rd \kappa <\infty$\}.

%[Part \ref{lem.prop.prior.4}]
%By the change of variables $\kappa=1/(g+1)$, 
%the prior on $\beta$ given by  \eqref{our.prior.beta},
%is rewritten as
%\begin{equation}
% \pi(\beta)=\frac{1}{(2\pi)^{p/2}}
%\int_0^\infty \psi(g)
%\exp\left(-\frac{\|\beta\|^2}{2g}\right)
%\frac{\rd g}{g},
%\end{equation}
%where
%\begin{equation}
% \psi(g)=
%\left(\frac{1}{g+1}\right)^{p/2+a}\left(\frac{g}{g+1}\right)^{b-p/2}
%h(1/(g+1)).
%\end{equation}
%Note 
%\begin{equation}
%\lim_{g\to 0} \frac{\psi(g)}{g^{-p/2+b}h(1)}=1,\quad
%\lim_{g\to \infty} \frac{\psi(g)}{g^{-p/2-a}h(1/g)}=1.
%\end{equation}
%Then, by Tauberian Theorem, we have
%%the asymptotic behavior of $\pi(\beta)$ is given by
%\begin{align}
% \lim_{\|\beta\|\to 0}\frac{\pi(\beta)}{\|\beta\|^{-p+2b}}&=h(1)\frac{\Gamma(p/2-b)2^{p/2-b}}{(2\pi)^{p/2}}\\
% \lim_{\|\beta\|\to \infty}\frac{\pi(\beta)}{\|\beta\|^{p+2a}h(1/\|\beta\|)}&=\frac{\Gamma(p/2+a)2^{p/2+a}}{(2\pi)^{p/2}},
%\end{align}
%which implies that the parameters $a$ and $b$ determine the behavior of $\pi(\beta)$ 
%at $\|\beta\|=\infty$ and $\|\beta\|=0$, respectively.

[Part \ref{lem.prop.prior.5}]
By the identity
\begin{equation}
 \|y-\beta\|^2+\frac{\kappa}{1-\kappa}\|\beta\|^2
=\frac{1}{1-\kappa}\left\|\beta-(1-\kappa)y\right\|^2
+\kappa\|y\|^2,
\end{equation}
the marginal density is
\begin{align}
 m(y)&=\int \frac{1}{(2\pi)^{p/2}}\exp\left(-\frac{\|y-\beta\|^2}{2}\right)\pi(\beta)\rd \beta \\
&=\iint \frac{1}{(2\pi)^{p/2}}\exp\left(-\frac{\|y-\beta\|^2}{2}\right)
\frac{1}{(2\pi)^{p/2}}\left(\frac{\kappa}{1-\kappa}\right)^{p/2}\notag\\
&\quad\times\exp\left(-\frac{\kappa}{1-\kappa}\frac{\|\beta\|^2}{2}\right)
\pi(\kappa)
\rd \beta\rd\kappa \notag\\
&=\frac{1}{(2\pi)^{p/2}}\int \exp\left(-\kappa\frac{\|y\|^2}{2}\right)
\kappa^{p/2}\pi(\kappa)\rd\kappa \notag\\
&=\frac{1}{(2\pi)^{p/2}}
\int_0^1\exp\left(-\kappa\frac{\|y\|^2}{2}\right)\kappa^{p/2+a-1}(1-\kappa)^{b-1}h(\kappa)\rd\kappa.\notag
\end{align}
From  $\exp(-\kappa\|y\|^2/2)\leq 1$, 
\begin{equation}
 m(y)\leq \frac{1}{(2\pi)^{p/2}}
\int_0^1\kappa^{p/2+a-1}(1-\kappa)^{b-1}h(\kappa)\rd\kappa.\label{marginal.0}
\end{equation}
The right hand side of \eqref{marginal.0} is finite (integrable)
if either $a>-p/2 $ or 
\{$a=-p/2$ and $\int_0^1 \kappa^{-1}h(\kappa)\rd \kappa <\infty$\}.
\end{proof}
By %As in the equation (3) of \cite{Polson-Scott-2012} or 
Tweedie's formula \citep{Efron-2011, Efron-2023-jjsd}, the Bayes estimator can be expressed as
\begin{equation}\label{hat.beta}
 \hat{\beta}
=y+\nabla \log m(y) 
=\left(1-\frac {\int_0^1\kappa^{p/2+1}\exp(-\kappa\|y\|^2/2)\pi(\kappa)\rd\kappa}
{\int_0^1\kappa^{p/2}\exp(-\kappa\|y\|^2/2)\pi(\kappa)\rd\kappa}
\right)y.
\end{equation}
%where $m(y) $ is given by \eqref{my.lem.prop.prior}, as
%\begin{equation}\label{my.1}
%%\begin{split}
% m(y)%&\propto %\int_0^1\kappa^{p/2}\exp(-\kappa\|y\|^2/2)\pi(\kappa)\rd\kappa\\
%= \frac{1}{(2\pi)^{p/2}}\int_0^1\kappa^{p/2+a-1}\exp(-\kappa\|y\|^2/2)(1-\kappa)^{b-1}h(\kappa)\rd\kappa.
%%\end{split} 
%\end{equation}
By Lemma \ref{lem.prop.prior}, the marginal density can be finite even when $\pi(\kappa)$ is improper.
In such case, the generalized Bayes estimator is still given by \eqref{hat.beta}.
In the next section, we investigate the minimaxity of both proper Bayes and generalized Bayes estimators.

\section{Minimaxity}
\label{sec.minimaxity}
Let 
\begin{equation}\label{capital.H}
H(\kappa)=\kappa\frac{h'(\kappa)}{h(\kappa)}, \quad 
H_1(\kappa)=\inf_{t\leq \kappa}H(t) \ \text{ and } \ H_2(\kappa)=H(\kappa)-H_1(\kappa).
\end{equation}
The definition of $H_1(\kappa)$ above and \eqref{hh} in \ref{A.1}
imply that $H_1(\kappa)$ is non-positive and monotone non-increasing.
Also, $H_2(\kappa)$ is non-negative and $H_1(0)= H_2(0)=0$ by \eqref{hh} and \eqref{capital.H}.
Then, we have the following result.
\begin{thm}\label{thm.minimax}
Suppose $-p/2\leq a<p/2-2$. The (generalized) Bayes estimator with respect to the prior \eqref{our.prior} is minimax if
\begin{equation}\label{thm.minimax.eq.1}
\frac{3p}{2}+a-\frac{p+2a+2+2\max_{\kappa\in[0,1]} H_2(\kappa)}{b}
+\min\left\{0,p/2+a+2+H_1(1)\right\}\geq 0.
% \begin{split}
%& \frac{p}{2}-a-2+
%2(b-1)\frac{\int_{0}^{1} \kappa^{p/2+a+1}(1-\kappa)^{b-1}h(\kappa) \rd \kappa}
%{\int_{0}^{1} \kappa^{p/2+a}(1-\kappa)^{b}h(\kappa) \rd \kappa}-2\max_{\kappa\in[0,1]}H_2(\kappa)
%\\
%&+
%\min\left(0,p/2+a+2+\frac{\int_{0}^{1}H_1(\kappa)\kappa^{p/2+a}(1-\kappa)^{b-1}h(\kappa)\rd \kappa}
%{\int_{0}^{1}\kappa^{p/2+a}(1-\kappa)^{b-1}h(\kappa)\rd \kappa}\right)\geq 0.
%\end{split}
\end{equation}
\end{thm}
In the proof, we utilize the following ``correlation inequality'' several times.
\begin{lem}\label{lem:cor}
 Suppose $f(x)$ and $g(x)$ are both monotone non-decreasing in $x$. 
Let $X$ be a continuous random variable. Then,
\begin{equation}
 E[f(X)g(X)]\geq E[f(X)]E[g(X)].
\end{equation}
\end{lem}
\begin{proof}[Proof of Theorem \ref{thm.minimax}]
Since the marginal density $m(y)$ given by \eqref{my.1} is spherically symmetric, 
let $m_*(\|y\|^2) \coloneqq m(y)$.
Then, the quadratic risk of the (generalized) Bayes estimator \eqref{hat.beta}
%\begin{equation}
% \hat{\beta}=y-2\frac{m'_*(\|y\|^2)}{m_*(\|y\|^2)}y
%\end{equation}
is
\begin{equation}
 \begin{split}
&  E\left[\|\hat{\beta}-\beta\|^2\right] =E\left[\left\|Y-2\frac{m'_*(\|Y\|^2)}{m_*(\|Y\|^2)}Y-\beta\right\|^2\right]\\
&=E\left[\|Y-\beta\|^2+4\left(\frac{m'_*(\|Y\|^2)}{m_*(\|Y\|^2)}\right)^2\|Y\|^2
-4\sum_{i=1}^p(Y_i-\beta_i)Y_i\frac{m'_*(\|Y\|^2)}{m_*(\|Y\|^2)}\right]\\
&=p+4E\left[\frac{m'_* (\|Y\|^2)}{m_* (\|Y\|^2)}
\Bigl(p-2\|Y\|^2\frac{m''_*(\|Y\|^2)}{-m'_*(\|Y\|^2)}
+\|Y\|^2\frac{-m'_*(\|Y\|^2)}{m_*(\|Y\|^2)}\Bigr)\right],
 \end{split}
\end{equation}
where the third equality follows from \cite{Stein-1974} identity.
Note $m'_*(\|y\|^2)\leq 0$. 
Thus, a sufficient condition for minimaxity %or $E[\|\hat{\beta}-\beta\|^2]\leq p$
is 
%\begin{equation}
%-\frac{\Delta m(y)}{\| \nabla m(y)\|}+\frac{1}{2}\frac{\| \nabla m(y)\|}{m(y)}
% \geq 0
%\end{equation}
%which is, for  \eqref{my.1}, rewritten as
\begin{equation}\label{suffi.cond.1}
 \begin{split}
 &p-\| y \|^{2}
\frac{ \int_{0}^{1}\kappa^{p/2+a+1}(1-\kappa)^{b-1}h(\kappa)\exp(-\kappa\| y\|^{2}/2) \rd \kappa}
{ \int_{0}^{1}\kappa^{p/2+a}(1-\kappa)^{b-1}h(\kappa)\exp(-\kappa\| y\|^{2}/2) \rd \kappa}\\
&\qquad 
+\frac{\| y \|^{2}}{2}\frac{ \int_{0}^{1}\kappa^{p/2+a}(1-\kappa)^{b-1}h(\kappa)\exp(-\kappa\| y\|^{2}/2) \rd \kappa}
{\int_{0}^{1}\kappa^{p/2+a-1}(1-\kappa)^{b-1}h(\kappa)\exp(-\kappa\| y\|^{2}/2) \rd \kappa}
\geq 0. 
\end{split}
\end{equation}
For $w=\|y\|^2/2$, the sufficient condition \eqref{suffi.cond.1} is equivalent to
\begin{equation}
 \Delta(w)\geq 0,
\end{equation}
where
\begin{equation}
 \begin{split}
\Delta(w) &=p-2w
\frac{ \int_{0}^{1}\kappa^{p/2+a+1}(1-\kappa)^{b-1}h(\kappa)\exp(w\{1-\kappa\}) \rd \kappa}
{ \int_{0}^{1}\kappa^{p/2+a}(1-\kappa)^{b-1}h(\kappa)\exp(w\{1-\kappa\} ) \rd \kappa}\\
&\quad
+w\frac{ \int_{0}^{1}\kappa^{p/2+a}(1-\kappa)^{b-1}h(\kappa)\exp(w\{1-\kappa\}) \rd \kappa}
{\int_{0}^{1}\kappa^{p/2+a-1}(1-\kappa)^{b-1}h(\kappa)\exp(w\{1-\kappa\}) \rd \kappa}.
\end{split}
\end{equation}
Note
\begin{equation}\label{key.derivative}
\frac{\rd }{\rd \kappa}\left\{-\exp(w\{1-\kappa\})+1 \right\}=w \exp(w\{1-\kappa\})
\end{equation}
and
\begin{equation}
 \left[\kappa^{p/2+a+1}(1-\kappa)^{b-1}h(\kappa)\left\{-\exp(w\{1-\kappa\})+1 \right\}\right]_0^1=0.
\end{equation}
Then we have
\begin{equation}\label{int.part.1}
 \begin{split}
 & w \int_{0}^{1}\kappa^{p/2+a+1}(1-\kappa)^{b-1}h(\kappa)\exp(w\{1-\kappa\}) \rd \kappa\\
&=(p/2+a+1)\int_{0}^{1}\kappa^{p/2+a}(1-\kappa)^{b-1}h(\kappa)\left\{\exp(w\{1-\kappa\})-1\right\} \rd \kappa \\
&\quad +\int_{0}^{1}\kappa^{p/2+a+1}(1-\kappa)^{b-1}h'(\kappa)\left\{\exp(w\{1-\kappa\})-1\right\} \rd \kappa \\
&\quad -(b-1)\int_{0}^{1}\kappa^{p/2+a+1}(1-\kappa)^{b-2}h(\kappa)\left\{\exp(w\{1-\kappa\})-1\right\} \rd \kappa. 
\end{split}
\end{equation}
%where the second equality follows from $\kappa=-(1-\kappa)+1 $.
Similarly we have
\begin{equation}\label{int.part.2}
 \begin{split}
 & w \int_{0}^{1}\kappa^{p/2+a}(1-\kappa)^{b-1}h(\kappa)\exp(w\{1-\kappa\}) \rd \kappa\\
&=(p/2+a)\int_{0}^{1}\kappa^{p/2+a-1}(1-\kappa)^{b-1}h(\kappa)\left\{\exp(w\{1-\kappa\})-1\right\} \rd \kappa \\
&\quad +\int_{0}^{1}\kappa^{p/2+a}(1-\kappa)^{b-1}h'(\kappa)\left\{\exp(w\{1-\kappa\})-1\right\} \rd \kappa \\
&\quad -(b-1)\int_{0}^{1}\kappa^{p/2+a}(1-\kappa)^{b-2}h(\kappa)\left\{\exp(w\{1-\kappa\})-1\right\} \rd \kappa .
\end{split}
\end{equation}
Recall $ H(\kappa)=\kappa h'(\kappa)/h(\kappa)$ and let
\begin{equation}
G(\kappa)=\kappa^{p/2+a-1}(1-\kappa)^{b-1}h(\kappa). 
\end{equation}
Then, by \eqref{int.part.1} and \eqref{int.part.2}, we have
\begin{equation}\label{D.D1.D2}
\Delta(w)=\frac{p}{2}-a-2+(1-b)
\Delta_1(w)+\Delta_2(w),
\end{equation}
where
\begin{equation}
 \begin{split}
 \Delta_1(w)&=
 \frac{\int_{0}^{1}\kappa(1-\kappa)^{-1}\left\{\exp(w\{1-\kappa\})-1\right\}
G(\kappa)\rd \kappa}
{\int_{0}^{1}\exp(w\{1-\kappa\} ) G(\kappa) \rd \kappa}\\
&\quad -2\frac{\int_{0}^{1}\kappa^2(1-\kappa)^{-1}\left\{\exp(w\{1-\kappa\})-1\right\}G(\kappa) \rd \kappa}
{\int_{0}^{1}\kappa \exp(w\{1-\kappa\} ) G(\kappa)\rd \kappa},
\end{split}
\end{equation}
\begin{align}
& \Delta_2(w)\notag\\
&= 2\frac{(p/2+a+1)\int_{0}^{1}\kappa G(\kappa)\rd \kappa }
{ \int_{0}^{1}\kappa \exp(w\{1-\kappa\} ) G(\kappa)\rd \kappa}
 -\frac{(p/2+a)\int_{0}^{1} G(\kappa)\rd \kappa}
{ \int_{0}^{1}\exp(w\{1-\kappa\} ) G(\kappa)\rd \kappa}\label{D2}\\
&\quad 
+\frac{\int_{0}^{1} H(\kappa)\{\exp(w\{1-\kappa\})-1\}G(\kappa)\rd \kappa }
{\int_{0}^{1} \exp(w\{1-\kappa\})G(\kappa)\rd \kappa}
-2\frac{\int_{0}^{1}\kappa H(\kappa)\{\exp(w\{1-\kappa\})-1\}G(\kappa)\rd \kappa }
{\int_{0}^{1}\kappa \exp(w\{1-\kappa\})G(\kappa)\rd \kappa}.\notag
\end{align}

[The lower bound of $\Delta_1(w)$]\quad 
Note the expansion,
\begin{equation}
 \exp(w\{1-\kappa\})-1=\sum_{j=1}^\infty\frac{w^j(1-\kappa)^j}{j!}.
\end{equation}
Then we have
\begin{equation}
\Delta_1(w) \geq 
-2\frac{\int_{0}^{1}\kappa^2(1-\kappa)^{-1}\left\{\exp(w\{1-\kappa\})-1\right\}G(\kappa) \rd \kappa}{\int_{0}^{1}\kappa \left\{\exp(w\{1-\kappa\})-1\right\}G(\kappa) \rd \kappa}.
\end{equation}
Since the correlation inequality (Lemma~\ref{lem:cor}) gives
\begin{equation}
 \int_{0}^{1}\frac{\kappa^2}{(1-\kappa)}(1-\kappa)^{j} G(\kappa)\rd \kappa 
\leq \frac{\int_{0}^{1}\kappa^2 G(\kappa) \rd \kappa}
{\int_{0}^{1}\kappa(1-\kappa) G(\kappa) \rd \kappa}
\int_{0}^{1}\kappa(1-\kappa)^{j} G(\kappa)\rd \kappa
\end{equation}
for $j\geq 1$,
we have
\begin{equation}
\frac{\int_{0}^{1}\kappa^2(1-\kappa)^{-1}\left\{\exp(w\{1-\kappa\})-1\right\}G(\kappa) \rd \kappa}{\int_{0}^{1}\kappa\left\{\exp(w\{1-\kappa\})-1\right\}G(\kappa) \rd \kappa}
\leq 
\frac{\int_{0}^{1}\kappa^2 G(\kappa) \rd \kappa}
{\int_{0}^{1}\kappa(1-\kappa) G(\kappa) \rd \kappa}
\end{equation}
and
\begin{equation}
 \Delta_1(w) \geq -2
\frac{\int_{0}^{1}\kappa^2 G(\kappa) \rd \kappa}
{\int_{0}^{1}\kappa(1-\kappa) G(\kappa) \rd \kappa}
=-2\frac{\int_{0}^{1} \kappa^{p/2+a+1}(1-\kappa)^{b-1}h(\kappa) \rd \kappa}
{\int_{0}^{1} \kappa^{p/2+a}(1-\kappa)^{b}h(\kappa) \rd \kappa}.\label{D1.final}
\end{equation}
Further an integration by parts gives
\begin{equation}
 \begin{split}
& b\int_{0}^{1} \kappa^{p/2+a+1}(1-\kappa)^{b-1}h(\kappa) \rd \kappa \\
&=\left[-\kappa^{p/2+a+1}(1-\kappa)^{b}h(\kappa)\right]_0^1 +(p/2+a+1)\int_{0}^{1} \kappa^{p/2+a}(1-\kappa)^{b}h(\kappa) \rd \kappa \\
&\quad +\int_{0}^{1} \kappa^{p/2+a}(1-\kappa)^{b}H(\kappa)h(\kappa) \rd \kappa \\
&\leq \left\{p/2+a+1+\max_{\kappa\in[0,1]} H_2(\kappa)\right\}\int_{0}^{1} \kappa^{p/2+a}(1-\kappa)^{b}h(\kappa) \rd \kappa,
\end{split}
\end{equation}
where the inequality follows from
\begin{equation}
 H(\kappa)=H_1(\kappa)+H_2(\kappa)\leq H_2(\kappa)\leq\max_{\kappa\in[0,1]} H_2(\kappa).
\end{equation}
Then
\begin{equation}
 \Delta_1(w) \geq -2\frac{p/2+a+1+\max_{\kappa\in[0,1]} H_2(\kappa)}{b}.
\end{equation}
[The lower bound of $\Delta_2(w)$]\quad 
The correlation inequality (Lemma~\ref{lem:cor}) gives
\begin{equation}
 \frac{\int_{0}^{1} G(\kappa)\rd \kappa}
{ \int_{0}^{1}\exp(w\{1-\kappa\} ) G(\kappa)\rd \kappa}
\leq \frac{\int_{0}^{1}\kappa G(\kappa)\rd \kappa }
{ \int_{0}^{1}\kappa \exp(w\{1-\kappa\} ) G(\kappa)\rd \kappa}
\end{equation}
and hence, for the first and second terms of \eqref{D2}, we have
\begin{equation}\label{D2.1}
 \begin{split}
& 2\frac{(p/2+a+1)\int_{0}^{1}\kappa G(\kappa)\rd \kappa }
{ \int_{0}^{1}\kappa \exp(w\{1-\kappa\} ) G(\kappa)\rd \kappa}
 -\frac{(p/2+a)\int_{0}^{1} G(\kappa)\rd \kappa}
{ \int_{0}^{1}\exp(w\{1-\kappa\} ) G(\kappa)\rd \kappa}\\
&\geq 
\frac{(p/2+a+2)\int_{0}^{1}\kappa G(\kappa)\rd \kappa }
{ \int_{0}^{1}\kappa \exp(w\{1-\kappa\} ) G(\kappa)\rd \kappa}.
\end{split}
\end{equation}
By \eqref{capital.H}, we have
\begin{equation}\label{H1H2}
 \begin{split}
& \frac{\int_{0}^{1} H(\kappa)\{\exp(w\{1-\kappa\})-1\}G(\kappa)\rd \kappa }
{\int_{0}^{1} \exp(w\{1-\kappa\})G(\kappa)\rd \kappa}
-2\frac{\int_{0}^{1}\kappa H(\kappa)\{\exp(w\{1-\kappa\})-1\}G(\kappa)\rd \kappa }
{\int_{0}^{1}\kappa \exp(w\{1-\kappa\})G(\kappa)\rd \kappa} \\
& =
\sum_{i=1}^2\left(\frac{\int_{0}^{1} H_i(\kappa)\{\exp(w\{1-\kappa\})-1\}G(\kappa)\rd \kappa }
{\int_{0}^{1} \exp(w\{1-\kappa\})G(\kappa)\rd \kappa}-2\frac{\int_{0}^{1}\kappa H_i(\kappa)\{\exp(w\{1-\kappa\})-1\}G(\kappa)\rd \kappa }
{\int_{0}^{1}\kappa \exp(w\{1-\kappa\})G(\kappa)\rd \kappa}\right).
\end{split}
\end{equation}
Recall $H_1(\kappa)$ is monotone non-increasing. Then,
by the correlation inequality (Lemma~\ref{lem:cor}), we have
\begin{equation}
 \frac{\int_{0}^{1} H_1(\kappa)\exp(w\{1-\kappa\})G(\kappa)\rd \kappa }
{\int_{0}^{1} \exp(w\{1-\kappa\})G(\kappa)\rd \kappa}
\geq 
\frac{\int_{0}^{1}\kappa H_1(\kappa)\exp(w\{1-\kappa\})G(\kappa)\rd \kappa }
{\int_{0}^{1}\kappa \exp(w\{1-\kappa\})G(\kappa)\rd \kappa}
\end{equation}
and hence
\begin{equation}\label{D2.2}
 \begin{split}
& \frac{\int_{0}^{1} H_1(\kappa)\{\exp(w\{1-\kappa\})-1\}G(\kappa)\rd \kappa }
{\int_{0}^{1} \exp(w\{1-\kappa\})G(\kappa)\rd \kappa}-2\frac{\int_{0}^{1}\kappa H_1(\kappa)\{\exp(w\{1-\kappa\})-1\}G(\kappa)\rd \kappa }
{\int_{0}^{1}\kappa \exp(w\{1-\kappa\})G(\kappa)\rd \kappa}\\
& \geq \frac{-\int_{0}^{1}\kappa H_1(\kappa)\exp(w\{1-\kappa\})G(\kappa)\rd \kappa 
+2 \int_{0}^{1}\kappa H_1(\kappa)G(\kappa)\rd \kappa }
{\int_{0}^{1}\kappa \exp(w\{1-\kappa\})G(\kappa)\rd \kappa}\\
& \geq \frac{\int_{0}^{1}\kappa H_1(\kappa)G(\kappa)\rd \kappa }
{\int_{0}^{1}\kappa \exp(w\{1-\kappa\})G(\kappa)\rd \kappa}.
\end{split}
\end{equation}
Further, for the $H_2$ part in \eqref{H1H2}, we have
\begin{equation}\label{D2.3}
 \begin{split}
 & \frac{\int_{0}^{1} H_2(\kappa)\{\exp(w\{1-\kappa\})-1\}G(\kappa)\rd \kappa }
{\int_{0}^{1} \exp(w\{1-\kappa\})G(\kappa)\rd \kappa}-2\frac{\int_{0}^{1}\kappa H_2(\kappa)\{\exp(w\{1-\kappa\})-1\}G(\kappa)\rd \kappa }
{\int_{0}^{1}\kappa \exp(w\{1-\kappa\})G(\kappa)\rd \kappa}\\
&\geq -2
\frac{\int_{0}^{1}\kappa H_2(\kappa)\exp(w\{1-\kappa\})G(\kappa)\rd \kappa}
{\int_{0}^{1}\kappa \exp(w\{1-\kappa\})G(\kappa)\rd \kappa} \\
&\geq -2\max_{\kappa\in[0,1]}H_2(\kappa). 
\end{split}
\end{equation}
Then, by \eqref{D2.1}, \eqref{D2.2}, and \eqref{D2.3}, we have
\begin{equation}\label{D2.final}
 \begin{split}
 \Delta_2(w)
&\geq \frac{\int_{0}^{1}\{p/2+a+2+H_1(\kappa)\}\kappa G(\kappa)\rd \kappa }
{\int_{0}^{1}\kappa \exp(w\{1-\kappa\})G(\kappa)\rd \kappa}-2\max_{\kappa\in[0,1]}H_2(\kappa)\\
%&\geq \min\left(0,p/2+a+2+\frac{\int_{0}^{1}\kappa H_1(\kappa)G(\kappa)\rd \kappa}
%{\int_{0}^{1}\kappa G(\kappa)\rd \kappa}\right)-2\max_{\kappa\in[0,1]}H_2(\kappa) \\
&\geq \min\left\{0,p/2+a+2+H_1(1)\right\}-2\max_{\kappa\in[0,1]}H_2(\kappa).
\end{split}
\end{equation}

Recall 
\begin{equation}
\Delta(w)= \frac{p}{2}-a-2+(1-b)
\Delta_1(w)+\Delta_2(w),
\end{equation}
as in \eqref{D.D1.D2}.
Further, as in \eqref{D1.final} and \eqref{D2.final}, 
the lower bound of $\Delta_1(w)$ and $\Delta_2(w)$ is negative and non-positive respectively.
Hence the sufficient condition $\Delta(w)\geq 0$ cannot hold if $ a\geq p/2-2$.
Let $a<p/2-2$. Then, by \eqref{D.D1.D2}, \eqref{D1.final} and \eqref{D2.final}, 
we have
\begin{equation}
\Delta(w)\geq 
\frac{3p}{2}+a-\frac{p+2a+2+2\max_{\kappa\in[0,1]} H_2(\kappa)}{b}
+\min\left\{0,p/2+a+2+H_1(1)\right\},
\end{equation}
which completes the proof.
\end{proof}

When $H(\kappa)$ is monotone non-increasing, it follows that
\begin{equation}\label{cor.1}
 H_1(\kappa)=H(\kappa) \text{ and }H_2(\kappa)\equiv 0
\end{equation}
for \eqref{capital.H}. 
When $H(\kappa)$ is monotone non-decreasing, it follows that
\begin{equation}\label{cor.2}
 H_1(\kappa)=0, \ H_2(\kappa)=H(\kappa)\text{ and }\max_{\kappa\in[0,1]} H_2(\kappa)=H(1),
\end{equation}
for \eqref{capital.H}. 
Then we have the following corollary.
\begin{corollary}\label{cor.minimax}
Suppose $-p/2\leq a<p/2-2$. 
\begin{enumerate}
 \item \label{cor.minimax.1}
Suppose $H(\kappa)$ is monotone non-increasing with $H(1)\geq -(p/2+a+2)$.
Then the (generalized) Bayes estimator with respect to the prior \eqref{our.prior} is minimax if 
\begin{align}
\frac{p+2a+2}{3p/2+a}\leq b<1.
\end{align}
\item \label{cor.minimax.2}
Suppose $H(\kappa)$ is monotone non-decreasing with $H(1)< (p/2-a-2)/2$.
Then the (generalized) Bayes estimator with respect to the prior \eqref{our.prior} is minimax if 
\begin{equation}
\frac{p+2a+2+2H(1)}{3p/2+a}\leq b<1.
\end{equation}
\end{enumerate}
\end{corollary}
%\begin{proof}
%For the monotone non-increasing $h(\kappa)$, the covariance inequality gives
%\begin{align}\label{cor.2}
%\frac{\int_{0}^{1} \kappa^{p/2+a+1}(1-\kappa)^{b-1}h(\kappa) \rd \kappa}
%{\int_{0}^{1} \kappa^{p/2+a}(1-\kappa)^{b}h(\kappa) \rd \kappa}
%\leq 
% \frac{\int_{0}^{1} \kappa^{p/2+a+1}(1-\kappa)^{b-1} \rd \kappa}
%{\int_{0}^{1} \kappa^{p/2+a}(1-\kappa)^{b} \rd \kappa}
%=\frac{p/2+a+1}{b}.
%\end{align}
%By \eqref{cor.1} and \eqref{cor.2}, the sufficient condition, \eqref{thm.minimax.eq.1}, 
%of Theorem \ref{thm.minimax} reduces to
%\begin{align}
% \frac{p}{2}-a-2+2(b-1)\frac{p/2+a+1}{b}\geq 0,
%\end{align}
%which is equivalent to
%\begin{align}
%b\geq \frac{p+2a+2}{3p/2+a}.
%\end{align}
%\end{proof}

%\section{Non-monotonic $\phi(w)$}
Before \cite{Fourdrinier-etal-1998}, minimaxity of estimators was typically shown through so-called 
\citeapos{Baranchik-1970} sufficient condition, which states that
the shrinkage estimator
\begin{equation}
 \left(1-\frac{\phi(\|y\|^2)}{\|y\|^2}\right)y
\end{equation}
is minimax if $\phi$ is monotone non-decreasing and $0\leq \phi(w)\leq 2(p-2)$.
For the (generalized) Bayes estimator with respect to the prior \eqref{our.prior}, the function $\phi$ is obtained from Tweedie's formula \eqref{hat.beta} as
\begin{equation}
 \phi(\|y\|^2)
=\|y\|^2\frac {\int_0^1\kappa^{p/2+1}\exp(-\kappa\|y\|^2/2)\kappa^{a-1}(1-\kappa)^{b-1}h(\kappa)\rd\kappa}
{\int_0^1\kappa^{p/2}\exp(-\kappa\|y\|^2/2)\kappa^{a-1}(1-\kappa)^{b-1}h(\kappa)\rd\kappa}. \label{phi}
\end{equation}
Then, we have the following result on the non-monotonicity of $\phi$, which means that we cannot employ \citeapos{Baranchik-1970} sufficient condition.  %with non-monotonicity.
\begin{lem}\label{lem.non.monotonicity}
\begin{enumerate}
 \item $\lim_{\|y\|^2\to\infty}\phi(\|y\|^2)=p+2a$.
\item \label{lem.non.monotonicity2}
Suppose
\begin{equation}\label{b.h}
1-b+ \liminf_{\kappa\to 0}\frac{h'(\kappa)}{h(\kappa)}>0.
\end{equation}
Then $ \phi$ is not monotone and approaches $p+2a$ from the above.
\end{enumerate}
\end{lem}
Suppose $H(\kappa)=\kappa h'(\kappa)/h(\kappa)$ 
is monotone non-decreasing. Since $H(0)=0$ by \ref{A.1} and \eqref{capital.H}, 
both $H(\kappa)$ and $h'(\kappa)/h(\kappa)=H(\kappa)/\kappa$ are non-negative for $\kappa\in[0,1]$.
Then we have
\begin{equation}
 \liminf_{\kappa\to 0}\frac{h'(\kappa)}{h(\kappa)}\geq 0 \ \text{ and } \
1-b+ \liminf_{\kappa\to 0}\frac{h'(\kappa)}{h(\kappa)}>0.
\end{equation}
Thus, from Part \ref{lem.non.monotonicity2} of Lemma~\ref{lem.non.monotonicity}, $\phi$ is non-monotone.
Hence, Part \ref{cor.minimax.2} of Corollary \ref{cor.minimax} provides the minimaxity of
(generalized) Bayes estimators with non-monotone $\phi$.
\begin{proof}[Proof of Lemma \ref{lem.non.monotonicity}]
Let $w=\|y\|^2/2$. Then
\begin{equation}
\phi(w) =2w\frac{\int_0^1\kappa^{p/2+a}\exp(-\kappa w)(1-\kappa)^{b-1}h(\kappa)\rd\kappa}
{\int_0^1\kappa^{p/2+a-1}\exp(-\kappa w)(1-\kappa)^{b-1}h(\kappa)\rd\kappa}.
\end{equation}
A Tauberian theorem (see, e.g., Theorem 13.5.4 in \cite{Feller-1971}) gives
\begin{equation}
\begin{split}
 \lim_{w\to\infty}
\frac{\int_0^1\kappa^{p/2+a}\exp(-\kappa w)(1-\kappa)^{b-1}h(\kappa)\rd\kappa}{w^{-p/2-a-1}h(1/w)}&=\Gamma(p/2+a+1)\\
 \lim_{w\to\infty}
\frac{\int_0^1\kappa^{p/2+a-1}\exp(-\kappa w)(1-\kappa)^{b-1}h(\kappa)\rd\kappa}{w^{-p/2-a}h(1/w)}&=\Gamma(p/2+a),
\end{split}
\end{equation}
which implies that
\begin{equation}
\lim_{\|y\|^2\to\infty} \phi(\|y\|^2)=p+2a.
\end{equation}

The way to approach to $p+2a$ determines $(1-\kappa)^{b-1}h(\kappa)$ as follows.
Under the condition \eqref{b.h},
there exists $0<\kappa_0<1$ and $ \epsilon>0$ such that
\begin{equation}\label{phi.1}
 \frac{1-b}{1-\kappa}+ \frac{h'(\kappa)}{h(\kappa)}\geq \epsilon
\end{equation}
for any $\kappa \in [0, \kappa_0]$.
An integration by parts with \eqref{key.derivative} and \eqref{int.part.1}
gives
\begin{equation}\label{int.part.phi}
\begin{split}
 & w\int_0^1\kappa^{p/2+a}\exp(-\kappa w)(1-\kappa)^{b-1}h(\kappa)\rd\kappa\\
& =(p/2+a)\int_0^1\kappa^{p/2+a-1}(1-\kappa)^{b-1}h(\kappa)\left\{\exp(-\kappa w)-\exp(-w)\right\}\rd\kappa\\
& \quad +\int_0^1\kappa^{p/2+a}\left\{(1-\kappa)^{b-1}h(\kappa)\right\}'\left\{\exp(-\kappa w)-\exp(-w)\right\}\rd\kappa\\
& =(p/2+a)\int_0^1\kappa^{p/2+a-1}(1-\kappa)^{b-1}h(\kappa)\exp(-\kappa w)\rd\kappa\\
&\quad -\frac{(p/2+a)\int_0^1\kappa^{p/2+a-1}(1-\kappa)^{b-1}h(\kappa)\rd\kappa+\int_0^{\epsilon_0}\kappa^{p/2+a}\left\{(1-\kappa)^{b-1}h(\kappa)\right\}'\rd\kappa}{\exp(w)} \\
& \quad +\int_0^{\epsilon_0}\kappa^{p/2+a}\left\{(1-\kappa)^{b-1}h(\kappa)\right\}'\exp(-\kappa w)\rd\kappa\\
&\quad +\int_{\epsilon_0}^1\kappa^{p/2+a}\left\{(1-\kappa)^{b-1}h(\kappa)\right\}'\left\{\exp(-\kappa w)-\exp(-w)\right\}\rd\kappa.
\end{split}
\end{equation}
For the last term, we have
\begin{equation}
\begin{split}
 \exp(-\kappa w)-\exp(-w)&= \exp(-\kappa w)\left\{1-\exp(-\{1-\kappa\} w)\right\}\\
&\leq \exp(-\kappa w)(1-\kappa)w 
\end{split} 
\end{equation}
and hence
\begin{equation}
\begin{split}
 & \left|\int_{\epsilon_0}^1\kappa^{p/2+a}\left\{(1-\kappa)^{b-1}h(\kappa)\right\}'\left\{\exp(-\kappa w)-\exp(-w)\right\}\rd\kappa \right|\\
 &\leq \frac{w}{\exp(\kappa_0 w)}\int_{\epsilon_0}^1\kappa^{p/2+a}\left|\left\{(1-\kappa)^{b-1}h(\kappa)\right\}'\right|
(1-\kappa)\rd \kappa.
\end{split}
\end{equation}
For the third term,
\begin{equation}
\begin{split}
 & \int_0^{\epsilon_0}\kappa^{p/2+a}\left\{(1-\kappa)^{b-1}h(\kappa)\right\}'\exp(-\kappa w)\rd\kappa\\
&= \int_0^{\epsilon_0}\kappa^{p/2+a}\left\{\frac{1-b}{1-\kappa}+\frac{h'(\kappa)}{h(\kappa)}
\right\}(1-\kappa)^{b-1}h(\kappa)\exp(-\kappa w)\rd\kappa \\
&\geq \epsilon\int_0^{\epsilon_0}\kappa^{p/2+a}(1-\kappa)^{b-1}h(\kappa)\exp(-\kappa w)\rd\kappa.
\end{split}
\end{equation}
Thus, by \eqref{phi.1}, we have
\begin{equation}
 \lim_{w\to\infty}
\frac{\int_0^{\epsilon_0}\kappa^{p/2+a}\left\{(1-\kappa)^{b-1}h(\kappa)\right\}'\exp(-\kappa w)\rd\kappa}{w^{-p/2-a-1}h(1/w)}\geq \epsilon \Gamma(p/2+a+1)
\end{equation}
and
\begin{equation}
\liminf_{\|y\|^2\to\infty} \|y\|^2\left\{\phi(\|y\|^2)-(p+2a)\right\}>0,
\end{equation}
which implies that $\phi$ is not monotone and approaches $p+2a$ from the above.
\end{proof}

\section{Choices of $a$, $b$ and $h(\kappa)$}
\label{sec.example}
In this section, some interesting choices of $a$, $b$ and $h(\kappa)$ in the prior \eqref{our.prior} satisfying
Corollary \ref{cor.minimax} are provided.

\subsection{A variant of half-Cauchy with $h(\kappa)\equiv 1$}
Recall that the prior \eqref{our.prior} coincides with the half-Cauchy prior when $a=b=1/2$ and $h(\kappa)\equiv 1$.
Although Figure 2 of \cite{Polson-Scott-2012} suggests the minimaxity of
the Bayes estimator under the half-Cauchy prior numerically, 
the choice $a=b=1/2$ is not included in Corollary \ref{cor.minimax} with $H(\kappa)\equiv 0$.
Among the prior \eqref{our.prior} with $a=b$ and $H(\kappa)\equiv 0$,
\begin{equation}
a=b=a_*=\frac{-3p+4+\sqrt{9p^2-8p+48}}{4}\in(0,1), \label{a_star}
\end{equation}
is the smallest choice satisfying Corollary \ref{cor.minimax} for $p\geq 7$.
Thus we find the choice 
\begin{equation}
 a=b= a_*, \quad h(\kappa)\equiv 1.
\end{equation}

\subsection{Log-adjusted prior}
\label{sec.ex1}
Let 
\begin{equation}
 h(\kappa)=\left\{1+c_1\log(1/\kappa)\right\}^{c_2},\label{prior.ex1}
\end{equation}
for $c_1>0$. Then
\begin{equation}
 H(\kappa)= \kappa\frac{h'(\kappa)}{h(\kappa)}=-c_2\frac{c_1}{1+c_1\log(1/\kappa)},
\end{equation}
which is monotone increasing/decreasing if $c_2$ is negative/positive, respectively. 
By Corollary \ref{cor.minimax} with $H(1)=-c_1c_2$, 
we have the following result.
\begin{corollary}\label{cor.minimax.example}
Let $-p/2\leq a<p/2-2$ and $h(\kappa)=\left\{1+c_1\log(1/\kappa)\right\}^{c_2}$ with $c_1>0$.
\begin{enumerate}
 \item \label{cor.minimax.example.1}
Suppose $0<c_2\leq (p/2+a+2)/c_1$.
Then the (generalized) Bayes estimator with respect to the prior \eqref{our.prior} is minimax if 
\begin{equation}
\frac{p+2a+2}{3p/2+a}\leq b<1.
\end{equation}
\item \label{cor.minimax.example.2}
Suppose $-(p/2-a-2)/(2c_1)<c_2< 0$.
Then the (generalized) Bayes estimator with respect to the prior \eqref{our.prior} is minimax if 
\begin{equation}
\frac{p+2a+2-2c_1c_2}{3p/2+a}\leq b<1.
\end{equation}
\end{enumerate}
\end{corollary}
As in Part \ref{lem.prop.prior.3} of Lemma \ref{lem.prop.prior}, 
the choice 
\begin{equation}\label{minimax.proper.boundary.1}
 a=0,\quad c_2=-2
\end{equation}
corresponds to a proper prior on the boundary between propriety and impropriety,
since
\begin{equation}
 \int_0^1 \frac{\rd\kappa}{\kappa\{1+c_1\log(1/\kappa)\}^2}=
\int_1^\infty \frac{\rd t}{(1+c_1t)^2}
=\left[\frac{1}{c_1(1+c_1t)}\right]_1^\infty=\frac{1}{c_1(1+c_1)}
\end{equation}
and
\begin{equation}
 \int_0^1 \frac{\rd\kappa}{\kappa\{1+c_1\log(1/\kappa)\}}=
\int_1^\infty \frac{\rd t}{1+c_1t}
=\left[\frac{\log(1+c_1t)}{c_1}\right]_1^\infty=\infty.
\end{equation}
Note that \cite{HIS.2001.08465} utilized a similar U-shape prior with $a=0$ and $c_1<-1$
for the analysis of sparse signals.
Let
\begin{equation}\label{minimax.proper.boundary.2}
c_1=\frac{p/2-a-2}{8}=\frac{p-4}{16}.
\end{equation}
Then, by Part \ref{cor.minimax.example.2} of Corollary \ref{cor.minimax.example},
the corresponding proper Bayes estimator is minimax if $p\geq 5$ and
\begin{equation}\label{minimax.proper.boundary.3}
\frac{5p+4}{6p}\leq b<1.
\end{equation}
Thus we find the choice 
\begin{equation}\label{choice.ex1}
 a=0,\quad b=\frac{5p+4}{6p}, \quad c_1= \frac{p-4}{16},\quad c_2=-2.
\end{equation}

%In Section \ref{sec:simulation}, 
%we develop an algorithm for sampling from the posterior distribution under Examples 1 and 2
%and investigate the properties of the corresponding Bayes estimators from the numerical viewpoint.

%By Lemma \ref{lem.non.monotonicity}, the corresponding $\phi$ is not monotone if $c_2< 0$.
%For the third term of \eqref{int.part.phi}, 
%\begin{equation}
%\begin{split}
%&  \int_0^{\epsilon_0}\kappa^{p/2+a}\left\{(1-\kappa)^{b-1}h(\kappa)\right\}'\exp(-\kappa w)\rd\kappa\\
%&=\int_0^{\epsilon_0}\kappa^{p/2+a}
%\left\{\frac{(1-b)}{1-\kappa}-\frac{c_1c_2}{\kappa}\frac{1}{1+c_1\log(1/\kappa)}\right\}(1-\kappa)^{b-1}h(\kappa)\exp(-\kappa w)\rd\kappa.
%\end{split}
%\end{equation}
%When $\kappa\to 0$, $-c_1c_2/\{\kappa(1+c_1\log(1/\kappa))\}$ is a major term and hence
%\begin{equation}
% \lim_{\|y\|^2\to \infty}\log(\|y\|^2)\left(\phi(\|y\|^2)-\{p+2a\}\right)=-2c_2.
%\end{equation}
%Recall that the origianl Half-Cauchy prior corresponds to $h(\kappa)\equiv 1$. In this case,
%$ (1-b)/(1-\kappa)$ is a major term and hence
%\begin{equation}
% \lim_{\|y\|^2\to \infty}\|y\|^2\left(\phi(\|y\|^2)-\{p+2a\}\right)=2(1-b)(p+2a).
%\end{equation}
%

\subsection{Hypergeometric inverted-beta prior}
Let 
\begin{equation}
 h(\kappa)=\left(1+c_3\kappa\right)^{c_4}\exp(d\kappa), \label{prior.ex2}
\end{equation}
for $c_3>0$. 
It can be viewed as a generalization of the ``hypergeometric inverted-beta prior'' by \cite{Polson-Scott-2012}, which corresponds to the prior \eqref{our.prior} with $h(\kappa)=(1+c\kappa)^{-1}\exp(d\kappa)$.
Then we have
\begin{equation}
%\begin{split}
 H(\kappa)= \kappa\frac{h'(\kappa)}{h(\kappa)}=d\kappa+c_4\frac{c_3\kappa}{1+c_3\kappa}
%\end{split}
\end{equation}
and
\begin{equation}
 H'(\kappa)=d+c_4\frac{c_3}{(1+c_3\kappa)^2}.
\end{equation}
%Thus, we have the following lemma.
\begin{lem}\label{lem:c3c4d}
 \begin{enumerate}
  \item The function $H(\kappa)$ is monotone non-increasing if
\begin{equation}
 \begin{cases}
  d\geq 0 \text{ and } c_4\leq -(c_3+1)^2d/c_3, \\
d<0 \text{ and }c_4\leq -d/c_3.
 \end{cases}
\end{equation}
  \item The function $H(\kappa)$ is monotone non-decreasing if
\begin{equation}
 \begin{cases}
d \geq 0  \text{ and }c_4 \geq  -d/c_3 \\
d< 0 \text{ and } c_4> -(c_3+1)^2d/c_3. 
 \end{cases}
\end{equation}
%\item Suppose $d>0$ and $ -(c_3+1)^2d/c_3<c_4<-d/c_3$. Then $H(\kappa)=H_1(\kappa)+H_2(\kappa)$
%where
%\begin{equation}
%H_1(\kappa)=\begin{cases}
%H(\kappa) & \text{ for }0\leq \kappa \leq  \\
%H(\kappa_*)	      & \text{ for }\kappa_* < \kappa \leq 1
%	     \end{cases}
%\end{equation}
%is monotone non-increasing and 
%$H_2(\kappa)=H(\kappa)-H_1(\kappa)$ is non-negative and 
%\begin{equation}
% H_2(\kappa)\leq H_2(1)=H(1)-H(\kappa_*)=\frac{d(1+c_3)^2-c_3c_4}{c_3(1+c_3)}.
%\end{equation}
 \end{enumerate}
\end{lem}
\begin{proof}
It is clear that $L(\kappa)$ is monotone non-decreasing/non-increasing for \{$d\geq 0$ and $c_4\geq 0$\}
and \{$d\leq 0$ and $c_4\leq 0$\}, respectively.
Further, for $dc_4<0$, we have
\begin{equation}\label{Lprime}
 \begin{split}
 H'(\kappa)&=\frac{d}{(1+c_3\kappa)^2}
\left((1+c_3\kappa)^2-\frac{-c_3c_4}{d}\right)\\
&=\frac{1+c_3\kappa+\sqrt{-c_3c_4/d}}{(1+c_3\kappa)^2/c_3}
d\left(\lambda-\kappa_*\right),
\end{split}
\end{equation}
where
\begin{equation}
 \kappa_*=\frac{\sqrt{-c_3c_4/d}-1}{c_3}.
\end{equation}
Then
$H'(\kappa)\geq 0$ follows
either if \{$d>0$ and $\kappa_*\leq 0$\} or if \{$d<0$ and $\kappa_*\geq 1$\}.
Similarly
$H'(\kappa)\leq 0$ follows
either if \{$d>0$ and $\kappa_*\geq 1$\} or if \{$d<0$ and $\kappa_*\leq 0$\}.
Thus we complete the proof.
 \end{proof}
By Lemma \ref{lem:c3c4d} and Corollary \ref{cor.minimax} with $H(1)=d+c_3c_4/(1+c_3)$, 
we have the following minimaxity result 
on the Bayes estimators under the hypergeometric inverted-beta priors, 
which was not fully investigated in \cite{Polson-Scott-2012}. 

\begin{corollary}\label{2.cor.minimax.example}
Let $-p/2\leq a<p/2-2$ and $h(\kappa)=\left(1+c_3\kappa\right)^{c_4}\exp(d\kappa)$ with $c_3>0$.
\begin{enumerate}
 \item \label{2.cor.minimax.example.1}
Suppose 
\begin{equation}
 \begin{cases}
  d\geq 0 \text{ and } c_4\leq -(c_3+1)^2d/c_3, \\
d<0 \text{ and }c_4\leq -d/c_3, \\
d+c_3c_4/(1+c_3)\geq -(p/2-a-2).
 \end{cases}
\end{equation}
Then the (generalized) Bayes estimator with respect to the prior \eqref{our.prior} is minimax if 
\begin{equation}
\frac{p+2a+2}{3p/2+a}\leq b<1.
\end{equation}
\item \label{2.cor.minimax.example.2}
Suppose 
\begin{equation}
 \begin{cases}
d \geq 0  \text{ and }c_4 \geq  -d/c_3, \\
d< 0 \text{ and } c_4> -(c_3+1)^2d/c_3, \\
d+c_3c_4/(1+c_3)\leq (p/2-a-2)/2.
 \end{cases}
\end{equation}
Then the (generalized) Bayes estimator with respect to the prior \eqref{our.prior} is minimax if 
\begin{equation}
\frac{p+2a+2+2\{d+c_3c_4/(1+c_3)\}}{3p/2+a}\leq b<1.
\end{equation}
\end{enumerate}
\end{corollary}

%\begin{lem}\label{lem.L}
%\begin{enumerate}
% \item For $d\geq 0$ and $c_4\geq 0$, $L(\kappa)$ is monotone non-decreasing.
% \item For $d\leq 0$ and $c_4\leq 0$, $L(\kappa)$ is monotone non-increasing.
%\item Suppose $d>0$ and $c_4<0$.
%\begin{enumerate}
% \item For $\kappa_*\geq 1$ or equivalently $c_4 \leq -(c_3+1)^2d/c_3$, 
%$L(\kappa)$ is monotone non-increasing.
% \item For $\kappa_*\leq 0$ or equivalently $0 > c_4 \geq -d/c_3$, 
%$L(\kappa)$ is monotone non-decreasing.
%\item For $\kappa_*\in(0,1)$, equivalently $-(c_3+1)^2d/c_3<c_4<-d/c_3$,
%\begin{align}
% L_1(\kappa)=\begin{cases}
%L(\kappa) & \text{ for }0\leq \kappa \leq \kappa_* \\
%L(\kappa_*)	      & \text{ for }\kappa_* < \kappa \leq 1
%	     \end{cases}
%\end{align}
%is monotone-nonincreasing. $L_2(\kappa)=L(\kappa)-L_1(\kappa)$ is non-negative and 
%$L_2(\kappa)\leq L_2(1)=L(1)-L(\kappa_*)$.
%\end{enumerate}
%\item Suppose $d<0$ and $c_4>0$.
%\begin{enumerate}
% \item For $\kappa_*\geq 1$ or equivalently $c_4\geq -(c_3+1)^2d/c_3$, 
%$L(\kappa)$ is monotone non-decreasing.
% \item For $\kappa_*\leq 0$ or equivalently $c_4\leq -d/c_3$, 
%$L(\kappa)$ is monotone non-increasing.
%\item For $\kappa_*\in(0,1)$, or $-d/c_3<c_4<-(c_3+1)^2d/c_3$,
%\begin{align}
% L_1(\kappa)=\begin{cases}
%0 & \text{ for }0\leq \kappa \leq \kappa_* \\
%L(\kappa)-L(\kappa_*)	      & \text{ for }\kappa < \kappa \leq 1
%	     \end{cases}
%\end{align}
%is monotone-nonincreasing. $L_2(\kappa)=L(\kappa)-L_1(\kappa)$ is non-negative and 
%$L_2(\kappa)\leq L(\kappa_*)$.
%\end{enumerate}
%\end{enumerate} 
%\end{lem}

\section{Simulation}
\label{sec:simulation}
\subsection{Algorithm for posterior sampling}
We develop an algorithm for sampling from the posterior $\pi(\beta,\kappa \mid y)$ corresponding to our prior $\pi(\beta,\kappa)$ in \eqref{hierarchical.2} and \eqref{our.prior}.
When $h(\kappa) \equiv 1$ (e.g. half-Cauchy prior), the Gibbs sampler can be written down in closed-form by using the fact that the beta-prime (or Pearson Type VI) distribution is given by an inverse-Gamma mixture of inverse-Gamma \citep{Polson-Scott-2012}.
However, it seems difficult to derive such a data-augmentation technique for general $h(\kappa)$.
Thus, we employ the Metropolis-within-Gibbs algorithm \citep{Robert-Casella-2004} here.

First, the conditional distribution $\pi (\beta \mid \kappa,y)$ is Gaussian with mean $(1-\kappa)y$ and covariance $(1-\kappa) I_p$ and thus easy to sample.
Next, the conditional distribution $\pi (\kappa \mid \beta,y)$ is 
\[
\pi (\kappa \mid \beta,y) \propto \kappa^{a+p/2-1} (1-\kappa)^{b-1} \exp \left( -\frac{\| y \|^2}{2} \kappa \right) h(\kappa), \quad 0 \leq \kappa \leq 1.
\]
To sample from it, we use the independent Metropolis--Hastings algorithm \citep{Robert-Casella-2004} with the proposal distribution given by ${\rm Beta}(\tilde{a},\tilde{b})$, where we set $\tilde{a}=\max(a,0.5)$ and $\tilde{b}=0.5$.
%Namely, we sample proposal $\kappa' \sim {\rm Beta} (,)$ and then accept it with probability $/$
In summary, our algorithm is given as follows:

\begin{enumerate}
	\item Set $\kappa_0=0.5$ and $t=1$
	
	\item Sample $\beta_t \sim {\rm N}_p ((1-\kappa_{t-1})y,(1-\kappa_{t-1})I_p)$
	
	\item Sample $\tilde{\kappa} \sim {\rm Beta}(\tilde{a},\tilde{b})$ and set
	\[
	\kappa_t = \begin{cases} \tilde{\kappa} & {\rm with\ probability} \ q \\ \kappa_{t-1} & {\rm otherwise} \end{cases}
	\]
	where
	\[
	q = \min \left( \left( \frac{\tilde{\kappa}}{\kappa} \right)^{a-\tilde{a}+p/2} \left( \frac{1-\tilde{\kappa}}{1-\kappa} \right)^{b-\tilde{b}} \exp \left( -\frac{\| y \|^2}{2} (\tilde{\kappa}-\kappa) \right) \frac{h(\tilde{\kappa})}{h(\kappa)}, 1 \right)
	\]
	
	\item Update $t$ to $t+1$ and return to 2
\end{enumerate}

%We set the number of Metropolis steps to 10 in the following.
%The acceptance rate was around 30 \%.
In computing the generalized Bayes estimator (posterior mean of $\beta$), we applied the Rao--Blackwellization technique for variance reduction \citep{Robert-Roberts-2021}.
Namely, we replaced step 2 of the above algorithm with $\beta_t = (1-\kappa_{t-1}) y$ and took average of sampled $\beta$ after burn-in.
%We confirmed the validity of this algorithm by checking the shrinkage factor (see Section~\ref{sec:shr_factor}).

In the following, we examine two priors.
The first one is 
\begin{align}
	\pi(\kappa) = \kappa^{a_*-1} (1-\kappa)^{a_*-1} \label{prior1}
\end{align}
where $a_* \in (0,1]$ is given in \eqref{a_star}.
From Corollary \ref{cor.minimax}, 
the Bayes estimator with respect to this prior is minimax when $p \geq 7$.
The second one is 
\begin{align}
	\pi(\kappa) = \kappa^{a-1} (1-\kappa)^{b-1} (1+c_1 \log(1/\kappa))^{c_2}, \label{prior2}
\end{align}
where $a=0, b=(5p+4)/(6p), c_1=(p-4)/16, c_2=-2$ as in \eqref{choice.ex1}.
From the discussion in Section~\ref{sec.ex1},  this prior is proper and the Bayes estimator with respect to it is minimax when $p \geq 5$.

\subsection{Risk comparison}
Figure~\ref{fig_mse} plots the quadratic risk of the Bayes estimators with respect to the priors \eqref{prior1} and \eqref{prior2} for $p=10$.
The risk of the James--Stein estimator is also plotted for comparison.
For all estimators, the risk is increasing with $\| \beta \|$ and converges to $p=10$ as $\| \beta \| \to \infty$, which indicates their minimaxity.
The  two Bayes estimators attain smaller risk than the James--Stein estimator when $\| \beta \| \leq 4$, which imply their advantage in sparse settings.

\begin{figure}[h]
	\centering
	\begin{tikzpicture}
		\begin{axis}[
			title={},
			xlabel={$\| \beta \|$}, xmin=0, xmax=10,
			ylabel={quadratic risk}, ymin=0, ymax=11,
			width=0.7\linewidth
			]
			\addplot[thick, color=black,
			filter discard warning=false, unbounded coords=discard
			] table {
         0    0.8718
0.1000    0.8327
0.2000    0.9324
0.3000    0.9748
0.4000    0.9921
0.5000    1.0307
0.6000    1.1439
0.7000    1.2320
0.8000    1.3093
0.9000    1.3791
1.0000    1.5537
1.1000    1.6267
1.2000    1.7856
1.3000    1.9502
1.4000    2.0757
1.5000    2.2687
1.6000    2.4102
1.7000    2.6207
1.8000    2.8640
1.9000    3.0259
2.0000    3.2268
2.1000    3.4668
2.2000    3.6218
2.3000    3.8889
2.4000    4.0025
2.5000    4.3083
2.6000    4.5709
2.7000    4.6818
2.8000    4.9727
2.9000    5.2375
3.0000    5.3401
3.1000    5.5527
3.2000    5.7104
3.3000    5.9061
3.4000    6.2748
3.5000    6.3376
3.6000    6.4874
3.7000    6.6802
3.8000    6.8269
3.9000    7.0442
4.0000    7.1286
4.1000    7.2511
4.2000    7.4561
4.3000    7.5987
4.4000    7.7927
4.5000    7.7447
4.6000    7.9832
4.7000    8.1125
4.8000    8.1027
4.9000    7.9768
5.0000    8.0274
5.1000    8.3023
5.2000    8.5281
5.3000    8.5222
5.4000    8.5532
5.5000    8.6362
5.6000    8.7351
5.7000    8.7155
5.8000    8.9813
5.9000    8.7020
6.0000    8.8050
6.1000    8.6825
6.2000    8.9560
6.3000    8.8857
6.4000    8.8389
6.5000    9.0649
6.6000    8.9163
6.7000    9.1659
6.8000    8.9161
6.9000    9.1327
7.0000    9.2679
7.1000    9.0214
7.2000    9.1478
7.3000    9.0496
7.4000    9.2300
7.5000    9.3483
7.6000    9.0861
7.7000    9.3051
7.8000    9.1738
7.9000    9.5554
8.0000    9.2758
8.1000    9.4274
8.2000    9.5099
8.3000    9.1300
8.4000    9.6334
8.5000    9.3528
8.6000    9.2772
8.7000    9.5828
8.8000    9.5467
8.9000    9.4017
9.0000    9.4164
9.1000    9.2993
9.2000    9.3212
9.3000    9.2318
9.4000    9.1696
9.5000    9.3743
9.6000    9.2944
9.7000    9.2321
9.8000    9.5387
9.9000    9.3970
10.0000    9.5380
			};
			\addplot[dashed, thick, color=black,
			filter discard warning=false, unbounded coords=discard
			] table {
				0    0.9992
0.1000    0.9999
0.2000    0.9987
0.3000    1.0505
0.4000    1.0767
0.5000    1.1358
0.6000    1.2228
0.7000    1.2900
0.8000    1.3932
0.9000    1.5063
1.0000    1.6176
1.1000    1.7367
1.2000    1.8935
1.3000    2.0357
1.4000    2.1933
1.5000    2.3567
1.6000    2.5228
1.7000    2.7018
1.8000    2.8825
1.9000    3.0858
2.0000    3.2564
2.1000    3.4856
2.2000    3.7129
2.3000    3.9219
2.4000    4.1043
2.5000    4.3348
2.6000    4.5263
2.7000    4.7607
2.8000    4.9489
2.9000    5.1376
3.0000    5.3501
3.1000    5.5159
3.2000    5.7354
3.3000    5.9133
3.4000    6.1076
3.5000    6.3101
3.6000    6.4452
3.7000    6.5827
3.8000    6.8142
3.9000    6.9238
4.0000    7.0402
4.1000    7.2470
4.2000    7.3432
4.3000    7.5729
4.4000    7.6150
4.5000    7.7205
4.6000    7.8069
4.7000    7.9545
4.8000    7.9800
4.9000    7.9874
5.0000    8.1427
5.1000    8.2483
5.2000    8.3395
5.3000    8.3936
5.4000    8.4137
5.5000    8.4514
5.6000    8.5583
5.7000    8.6252
5.8000    8.6216
5.9000    8.7200
6.0000    8.6897
6.1000    8.7281
6.2000    8.7622
6.3000    8.8743
6.4000    8.9087
6.5000    8.9216
6.6000    8.9167
6.7000    8.9424
6.8000    8.9458
6.9000    8.9878
7.0000    9.0561
7.1000    9.0110
7.2000    9.1328
7.3000    9.0250
7.4000    9.1211
7.5000    9.1262
7.6000    9.1035
7.7000    9.1865
7.8000    9.1684
7.9000    9.2336
8.0000    9.2271
8.1000    9.3278
8.2000    9.3694
8.3000    9.2771
8.4000    9.3049
8.5000    9.2550
8.6000    9.3187
8.7000    9.3647
8.8000    9.3417
8.9000    9.3826
9.0000    9.3448
9.1000    9.4414
9.2000    9.3565
9.3000    9.2981
9.4000    9.5018
9.5000    9.3667
9.6000    9.4243
9.7000    9.4495
9.8000    9.4770
9.9000    9.4588
10.0000    9.5176
			};
			\addplot[dotted, thick, color=black,
			filter discard warning=false, unbounded coords=discard
			] table {
				0    2.0643
0.1000    2.0238
0.2000    2.0129
0.3000    2.0943
0.4000    2.0854
0.5000    2.1452
0.6000    2.3016
0.7000    2.3728
0.8000    2.4821
0.9000    2.6238
1.0000    2.6979
1.1000    2.8935
1.2000    3.0315
1.3000    3.1940
1.4000    3.3660
1.5000    3.5265
1.6000    3.6476
1.7000    3.8338
1.8000    4.0043
1.9000    4.2403
2.0000    4.3383
2.1000    4.5358
2.2000    4.7339
2.3000    4.9289
2.4000    5.0706
2.5000    5.2278
2.6000    5.3715
2.7000    5.5797
2.8000    5.7109
2.9000    5.8095
3.0000    5.9909
3.1000    6.0787
3.2000    6.2575
3.3000    6.3300
3.4000    6.5072
3.5000    6.6471
3.6000    6.7043
3.7000    6.8000
3.8000    6.9938
3.9000    7.0513
4.0000    7.0939
4.1000    7.2885
4.2000    7.3353
4.3000    7.5518
4.4000    7.5419
4.5000    7.6209
4.6000    7.6679
4.7000    7.8093
4.8000    7.7978
4.9000    7.8156
5.0000    7.9201
5.1000    8.0450
5.2000    8.1056
5.3000    8.1752
5.4000    8.1803
5.5000    8.2147
5.6000    8.3200
5.7000    8.3763
5.8000    8.3839
5.9000    8.4872
6.0000    8.4460
6.1000    8.4965
6.2000    8.5356
6.3000    8.6412
6.4000    8.6904
6.5000    8.7179
6.6000    8.7185
6.7000    8.7568
6.8000    8.7574
6.9000    8.8053
7.0000    8.8821
7.1000    8.8443
7.2000    8.9490
7.3000    8.8518
7.4000    8.9582
7.5000    8.9692
7.6000    8.9547
7.7000    9.0281
7.8000    9.0304
7.9000    9.0910
8.0000    9.0899
8.1000    9.1976
8.2000    9.2333
8.3000    9.1587
8.4000    9.1732
8.5000    9.1411
8.6000    9.2097
8.7000    9.2475
8.8000    9.2263
8.9000    9.2741
9.0000    9.2249
9.1000    9.3364
9.2000    9.2478
9.3000    9.2085
9.4000    9.4006
9.5000    9.2739
9.6000    9.3336
9.7000    9.3586
9.8000    9.3952
9.9000    9.3809
10.0000    9.4267
			};
		\end{axis}
	\end{tikzpicture} 
	\caption{Comparison of quadratic risk for $p=10$. solid: Bayes estimator with respect to the prior \eqref{prior1}, dashed: Bayes estimator with respect to the prior \eqref{prior2}, dotted: James--Stein estimator.}
	\label{fig_mse}
\end{figure}
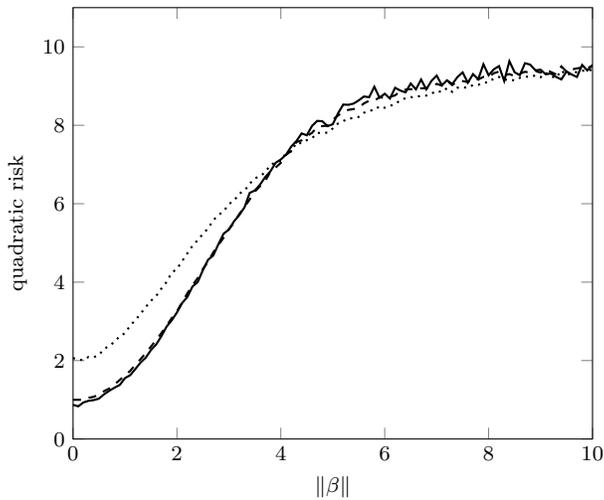

\subsection{Shrinkage factor}\label{sec:shr_factor}
%For a rotation-invariant estimator $\hat{\beta}$, its shrinkage factor $\phi(r)$ is defined by
%\[
%\hat{\beta} = \left( 1-\frac{\phi(\| y \|)}{\| y \|^2} \right) y.
%\]
%\cite{Maruyama-1998} studied the relation between shrinkage factors and (in)admissibility.
Figure~\ref{fig_shr} plots the shrinkage factors $\phi(\| y \|^2)$ in \eqref{phi} of Bayes estimators with respect to the priors \eqref{prior1} and \eqref{prior2} for $p=10$.
From Lemma~\ref{lem.non.monotonicity}, $\phi(\| y \|^2)$ for the prior \eqref{prior1} is non-monotone and converges to $p + 2 a_* \approx 11.7309$ from the above.
Similarly, $\phi(\| y \|^2)$ for the prior \eqref{prior2} is non-monotone and converges to $p = 10$ from the above.
The figure is compatible with these results and indicates that our sampling algorithm converged properly.

\begin{figure}[h]
	\centering
	\begin{tikzpicture}
		\begin{axis}[
			title={},
			xlabel={$\| y \|$}, xmin=0, xmax=10,
			ylabel={$\phi (\| y\|^2)$}, ymin=0, ymax=14,
			width=0.7\linewidth
			]
			\addplot[thick, color=black,
			filter discard warning=false, unbounded coords=discard
			] table {
0.1000    0.0087
0.2000    0.0348
0.3000    0.0784
0.4000    0.1393
0.5000    0.2175
0.6000    0.3126
0.7000    0.4254
0.8000    0.5546
0.9000    0.7009
1.0000    0.8636
1.1000    1.0430
1.2000    1.2385
1.3000    1.4505
1.4000    1.6779
1.5000    1.9209
1.6000    2.1784
1.7000    2.4488
1.8000    2.7348
1.9000    3.0353
2.0000    3.3498
2.1000    3.6745
2.2000    4.0082
2.3000    4.3567
2.4000    4.7109
2.5000    5.0728
2.6000    5.4464
2.7000    5.8209
2.8000    6.2003
2.9000    6.5828
3.0000    6.9698
3.1000    7.3506
3.2000    7.7321
3.3000    8.1138
3.4000    8.4777
3.5000    8.8341
3.6000    9.1813
3.7000    9.5084
3.8000    9.8134
3.9000   10.1147
4.0000   10.3736
4.1000   10.6303
4.2000   10.8474
4.3000   11.0556
4.4000   11.2224
4.5000   11.3717
4.6000   11.4959
4.7000   11.6090
4.8000   11.6889
4.9000   11.7563
5.0000   11.8073
5.1000   11.8459
5.2000   11.8548
5.3000   11.8940
5.4000   11.8889
5.5000   11.9047
5.6000   11.9070
5.7000   11.8943
5.8000   11.9078
5.9000   11.9038
6.0000   11.8802
6.1000   11.8683
6.2000   11.8613
6.3000   11.8567
6.4000   11.8475
6.5000   11.8526
6.6000   11.8263
6.7000   11.8463
6.8000   11.8394
6.9000   11.8398
7.0000   11.8239
7.1000   11.7964
7.2000   11.8185
7.3000   11.8036
7.4000   11.8260
7.5000   11.7958
7.6000   11.8136
7.7000   11.8063
7.8000   11.7921
7.9000   11.7786
8.0000   11.8051
8.1000   11.8018
8.2000   11.7800
8.3000   11.7915
8.4000   11.7885
8.5000   11.7953
8.6000   11.8125
8.7000   11.7891
8.8000   11.7719
8.9000   11.7789
9.0000   11.7510
9.1000   11.7615
9.2000   11.7749
9.3000   11.7773
9.4000   11.8029
9.5000   11.7548
9.6000   11.7686
9.7000   11.7713
9.8000   11.7991
9.9000   11.7939
10.0000   11.7521
			};
			\addplot[dashed, thick, color=black,
			filter discard warning=false, unbounded coords=discard
			] table {
				0.1000    0.0086
0.2000    0.0345
0.3000    0.0775
0.4000    0.1379
0.5000    0.2154
0.6000    0.3094
0.7000    0.4206
0.8000    0.5491
0.9000    0.6945
1.0000    0.8549
1.1000    1.0322
1.2000    1.2251
1.3000    1.4349
1.4000    1.6589
1.5000    1.8979
1.6000    2.1523
1.7000    2.4199
1.8000    2.7026
1.9000    2.9959
2.0000    3.3029
2.1000    3.6196
2.2000    3.9507
2.3000    4.2871
2.4000    4.6294
2.5000    4.9935
2.6000    5.3489
2.7000    5.7176
2.8000    6.0848
2.9000    6.4504
3.0000    6.8261
3.1000    7.1952
3.2000    7.5506
3.3000    7.9022
3.4000    8.2483
3.5000    8.5849
3.6000    8.8893
3.7000    9.2038
3.8000    9.4950
3.9000    9.7518
4.0000    9.9844
4.1000   10.2177
4.2000   10.4133
4.3000   10.5714
4.4000   10.7105
4.5000   10.8434
4.6000   10.9505
4.7000   11.0375
4.8000   11.0785
4.9000   11.1346
5.0000   11.1613
5.1000   11.1746
5.2000   11.1866
5.3000   11.1974
5.4000   11.2008
5.5000   11.2005
5.6000   11.1659
5.7000   11.1597
5.8000   11.1635
5.9000   11.1388
6.0000   11.1324
6.1000   11.1199
6.2000   11.1044
6.3000   11.0830
6.4000   11.0881
6.5000   11.0837
6.6000   11.0631
6.7000   11.0607
6.8000   11.0441
6.9000   11.0092
7.0000   11.0215
7.1000   11.0034
7.2000   10.9861
7.3000   10.9892
7.4000   11.0005
7.5000   10.9777
7.6000   10.9733
7.7000   10.9456
7.8000   10.9571
7.9000   10.9440
8.0000   10.9501
8.1000   10.9250
8.2000   10.9260
8.3000   10.9282
8.4000   10.9069
8.5000   10.9058
8.6000   10.8741
8.7000   10.8794
8.8000   10.9019
8.9000   10.8787
9.0000   10.8973
9.1000   10.8699
9.2000   10.8720
9.3000   10.8551
9.4000   10.8744
9.5000   10.8572
9.6000   10.8622
9.7000   10.8620
9.8000   10.8596
9.9000   10.8266
10.0000   10.8182
			};
	\addplot[dotted, thick, color=black,
	filter discard warning=false, unbounded coords=discard
	] table {
		0 8
		10 8
	};
		\end{axis}
	\end{tikzpicture} 
	\caption{Comparison of shrinkage factors for $p=10$. solid: Bayes estimator with respect to prior \eqref{prior1}, dashed: Bayes estimator with respect to prior \eqref{prior2}. dotted: James--Stein estimator.}
	\label{fig_shr}
\end{figure}
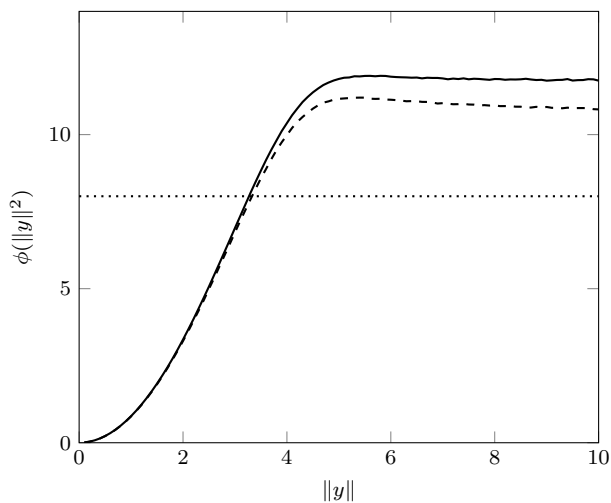

\subsection{Prior densities}
Figure~\ref{fig_prior} plots the log-prior densities of $\kappa$.
The prior \eqref{prior1} shows a symmetric U-shape with the same divergence speed at $\kappa=0$ and $\kappa=1$.
On the other hand, the prior \eqref{prior2} has faster divergence at $\kappa=0$ than at $\kappa=1$.
Note that \cite{carvalho2010horseshoe} used a similar plot to compare several types of shrinkage priors.

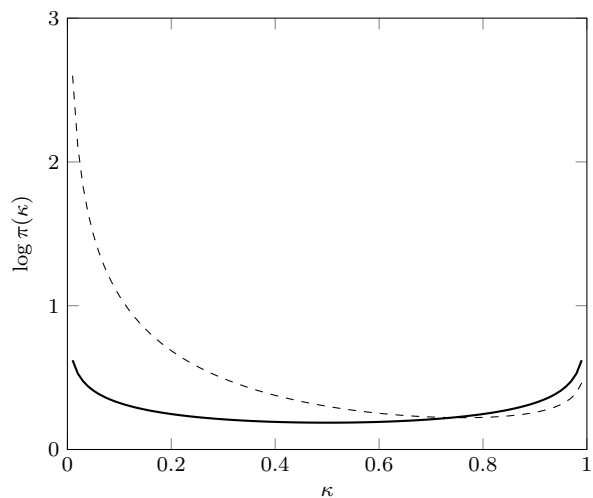
\begin{figure}
	\centering
	\begin{tikzpicture}
		\begin{axis}[
			title={},
			xlabel={$\kappa$}, xmin=0, xmax=1,
			ylabel={$\log \pi(\kappa)$}, ymin=0, ymax=3,
			width=0.7\linewidth
			]
			\addplot[thick, color=black,
filter discard warning=false, unbounded coords=discard
] table {

0.0100    0.6209
0.0200    0.5290
0.0300    0.4759
0.0400    0.4386
0.0500    0.4099
0.0600    0.3868
0.0700    0.3675
0.0800    0.3510
0.0900    0.3367
0.1000    0.3240
0.1100    0.3126
0.1200    0.3025
0.1300    0.2932
0.1400    0.2848
0.1500    0.2771
0.1600    0.2700
0.1700    0.2635
0.1800    0.2574
0.1900    0.2518
0.2000    0.2466
0.2100    0.2417
0.2200    0.2371
0.2300    0.2329
0.2400    0.2289
0.2500    0.2252
0.2600    0.2217
0.2700    0.2185
0.2800    0.2155
0.2900    0.2126
0.3000    0.2100
0.3100    0.2075
0.3200    0.2052
0.3300    0.2030
0.3400    0.2010
0.3500    0.1992
0.3600    0.1975
0.3700    0.1959
0.3800    0.1945
0.3900    0.1932
0.4000    0.1920
0.4100    0.1909
0.4200    0.1900
0.4300    0.1892
0.4400    0.1885
0.4500    0.1879
0.4600    0.1874
0.4700    0.1870
0.4800    0.1867
0.4900    0.1866
0.5000    0.1865
0.5100    0.1866
0.5200    0.1867
0.5300    0.1870
0.5400    0.1874
0.5500    0.1879
0.5600    0.1885
0.5700    0.1892
0.5800    0.1900
0.5900    0.1909
0.6000    0.1920
0.6100    0.1932
0.6200    0.1945
0.6300    0.1959
0.6400    0.1975
0.6500    0.1992
0.6600    0.2010
0.6700    0.2030
0.6800    0.2052
0.6900    0.2075
0.7000    0.2100
0.7100    0.2126
0.7200    0.2155
0.7300    0.2185
0.7400    0.2217
0.7500    0.2252
0.7600    0.2289
0.7700    0.2329
0.7800    0.2371
0.7900    0.2417
0.8000    0.2466
0.8100    0.2518
0.8200    0.2574
0.8300    0.2635
0.8400    0.2700
0.8500    0.2771
0.8600    0.2848
0.8700    0.2932
0.8800    0.3025
0.8900    0.3126
0.9000    0.3240
0.9100    0.3367
0.9200    0.3510
0.9300    0.3675
0.9400    0.3868
0.9500    0.4099
0.9600    0.4386
0.9700    0.4759
0.9800    0.5290
0.9900    0.6209
};
			\addplot[dashed, color=black,
			filter discard warning=false, unbounded coords=discard
			] table {
	0.0100    2.5998
0.0200    2.1080
0.0300    1.8308
0.0400    1.6396
0.0500    1.4948
0.0600    1.3790
0.0700    1.2830
0.0800    1.2013
0.0900    1.1305
0.1000    1.0682
0.1100    1.0128
0.1200    0.9629
0.1300    0.9178
0.1400    0.8765
0.1500    0.8387
0.1600    0.8039
0.1700    0.7716
0.1800    0.7415
0.1900    0.7135
0.2000    0.6873
0.2100    0.6627
0.2200    0.6396
0.2300    0.6178
0.2400    0.5973
0.2500    0.5778
0.2600    0.5594
0.2700    0.5419
0.2800    0.5253
0.2900    0.5095
0.3000    0.4945
0.3100    0.4801
0.3200    0.4664
0.3300    0.4534
0.3400    0.4409
0.3500    0.4290
0.3600    0.4176
0.3700    0.4067
0.3800    0.3962
0.3900    0.3862
0.4000    0.3767
0.4100    0.3675
0.4200    0.3587
0.4300    0.3502
0.4400    0.3422
0.4500    0.3344
0.4600    0.3270
0.4700    0.3199
0.4800    0.3131
0.4900    0.3066
0.5000    0.3003
0.5100    0.2944
0.5200    0.2887
0.5300    0.2833
0.5400    0.2781
0.5500    0.2731
0.5600    0.2684
0.5700    0.2640
0.5800    0.2597
0.5900    0.2557
0.6000    0.2519
0.6100    0.2484
0.6200    0.2450
0.6300    0.2419
0.6400    0.2390
0.6500    0.2363
0.6600    0.2338
0.6700    0.2315
0.6800    0.2295
0.6900    0.2276
0.7000    0.2260
0.7100    0.2246
0.7200    0.2235
0.7300    0.2225
0.7400    0.2218
0.7500    0.2214
0.7600    0.2212
0.7700    0.2213
0.7800    0.2217
0.7900    0.2224
0.8000    0.2234
0.8100    0.2247
0.8200    0.2264
0.8300    0.2284
0.8400    0.2309
0.8500    0.2339
0.8600    0.2374
0.8700    0.2415
0.8800    0.2462
0.8900    0.2517
0.9000    0.2581
0.9100    0.2656
0.9200    0.2744
0.9300    0.2848
0.9400    0.2973
0.9500    0.3128
0.9600    0.3323
0.9700    0.3584
0.9800    0.3963
0.9900    0.4630
			};
		\end{axis}
	\end{tikzpicture}
	\caption{Comparison of log-prior densities $\log \pi(\kappa)$ for $p=10$. solid: prior \eqref{prior1}, dashed: prior \eqref{prior2}}
	\label{fig_prior}
\end{figure}

\end{document}